\documentclass[11pt]{amsart}
\usepackage[utf8]{inputenc}

\usepackage{amsfonts, amsthm, amsmath, amssymb,bm}
\usepackage{amscd}
\usepackage{mathtools}

\usepackage{enumerate}

\newcommand{\rmod}{\!\!\!\!\pmod}

\newcommand{\Zz}{\mathbb{Z}}

\newcommand{\Cc}{\mathbb{C}}
\newcommand{\Hh}{\mathbb{H}}

\newtheorem{thm}{Theorem}[section]
\newtheorem{lem}[thm]{Lemma}

\newtheorem{prop}[thm]{Proposition}
\newtheorem{cor}[thm]{Corollary}
\newtheorem{rmk}{Remark}[section]

\title[Subconvexity for $\rm{GL}(3)$ L-functions]{On the subconvexity estimate for self-dual $\rm{GL}(3)$ L-functions in the $t$-aspect}
\author{Ramon M. Nunes} \thanks{This work is supported by the DFG-SNF lead agency program grant 200021L-153647.}
\date{\today}

\usepackage[square,sort,comma,numbers]{natbib}

\begin{document}

\maketitle

\begin{abstract}
We improve on the subconvexity bound for self-dual $\rm{GL}(3)$ $L$-functions in the $t$-aspect. Previous results were obtained by Li and by Mckee, Sun and Ye.
\end{abstract}

\section{Introduction}

In this paper we prove a subconvexity bound for certain degree 3 $L$-functions. Let $\phi$ be a self-dual Hecke-Maass form for $\rm{SL}(3,\Zz)$, or equivalently, let $\phi$ be the symmetric square lift of a Maass form for $\rm{SL}(2,\Zz)$. Then we have the following upper bound
$$
L(1/2+it,\phi)\ll_{\epsilon,\phi} t^{5/8+\epsilon}.
$$
Previous results were obtained by Li \citep{Li2011bounds} and more recentely, by Mckee, Sun and Ye \citep{mckee2015improved} who had $\ll t^{11/16+\epsilon}$ and $\ll t^{2/3+\epsilon}$ respectively. It is also worth mentioning Munshi's work \citep{munshi2015circle}, where he proved the bound $\ll t^{11/16+\epsilon}$ on a more general setting where the forms are not necessarily self-dual.

A common feature of the results in \citep{Li2011bounds}, \citep{mckee2015improved} and those of this paper is that they are deduced from an average result over the spectrum of the laplacian on $\rm{SL}(2,\Zz)\backslash \Hh$, where $\Hh$ is the upper half plane of complex numbers with positive imaginary part with the usual action of $\rm{SL}(2,\Zz)$. It is crucial for the result to work in the end that we have the positivity of certain $\rm{GL}(3)\times\rm{GL}(2)$ $L$-functions. Let $f$ be a Hecke-Maass form for $\rm{SL}(2,\Zz)$, then we have the inequality
$$
L(1/2,\phi\times f)\geq 0.
$$
This result follows from the work of Lapid \citep{lapid2003non}.

Munshi follows a different path where he does not need this positivity result. In fact these $\rm{GL}(3)\times \rm{GL}(2)$ $L$-functions do not even appear in his work.

\subsection{Statement of the main result}

We start by stating the average result from which we deduce the subconvexity bound. We refer the reader to section \ref{lfunctions} for a precise definition of the $L$-functions involved in it.

Let $\{f_j\}_j=\mathcal{B}$ be an orthonormal basis of Hecke-Maass forms for $\rm{SL}(2,\Zz)$, where $f_j$ is an eigenform for the Laplacian with eigenvalue $\frac14+t_j^2$.

\begin{thm}\label{GL3}
For every $\epsilon>0$, $T\geq 1$ and $\Delta\geq T^{\epsilon}$, we have the inequality
\begin{equation*}
\sideset{}{^{\prime}}\sum_{\substack{f_j\in \mathcal{B}\\T-\Delta<t_j\leq T+\Delta}}L\left(1/2,\phi\times f_j\right)+\frac{1}{4\pi}\int_{T-\Delta}^{T+\Delta}\left|L\left(1/2+it,\phi\right)\right|^2dt \ll \Delta T^{5/4+\epsilon},
\end{equation*}
where the symbol $\sideset{}{^{\prime}}\sum$ means that we only sum over the even forms.
\end{thm}

By taking $\Delta=T^{\epsilon}$ and using the positivity of $L\left(\frac{1}{2},\phi\times f_j\right)$, we deduce

\begin{cor}\label{subconv}
We have the following bounds:
$$
L\left(1/2,\phi\times f_j\right)\ll_{\phi} t_j^{5/4+\epsilon}\text{ and }L\left(1/2+it,\phi\right)\ll_{\phi} t^{5/8+\epsilon}.
$$
\end{cor}

In a related paper \citep{Blomer2012twisted}, Blomer proved a bound for quadratic twists of $L$-functions such as the ones considered in the present paper. Let $q$ be a prime number. Suppose that $f$ is a primitive Hecke-Maass form of level dividing $q$ and let $\phi$ be as above, then he proved the inequalities
\begin{equation}\label{boundsblomer}
L\left(\frac{1}{2},\phi\times f\times \chi_q\right)\ll_{\phi,f} q^{5/4+\epsilon}\text{ and }L\left(\frac{1}{2}+it,\phi\times \chi_q\right)\ll_{\phi,t} q^{5/8+\epsilon},
\end{equation}
where $\chi_q$ denotes the non-trivial quadratic character of conductor $q$.

We remark that the exponents in Corollary \ref{subconv} match exactly those of \eqref{boundsblomer}. This is not a coincidence. In fact, our method can be seen as the archimedean analog of that of \citep{Blomer2012twisted}. The way in which we prepare the ground in order to use the archimedean version of the large sieve is inspired by ideas of Young \citep{Young2014weyl}, who proved the hybrid bound
$$
L\left(\frac{1}{2}+it,\chi_q\right)\ll (tq)^{1/6+\epsilon}.
$$
Young's method generalise the work of Conrey and Iwaniec \citep{CI2000} and has the nice feature that it treats the $t$ and $q$ aspect on the same footing.

With some considerable extra effort, we believe that the techniques in this paper should allow for proving the hybrid bound 
$$
L(1/2+it,\phi\times \chi_q)\ll(tq)^{5/8+\epsilon},
$$
thus improving the main result of a recent preprint by Huang \citep{huang2016hybrid}. In order to simplify the exposition, we decided to stick to this more restrictive case.

Finally, we would like to mention that the idea of using Young's method in this $\rm{GL}(3)$ context was already present in Huang's paper. Nevertheless he did not improve on the exponent $2/3$ by Mckee, Sun and Ye. The main reason for this is that in his work (as was the case in \citep{Li2011bounds} and \citep{mckee2015improved}), Huang looks for upper bounds for the the sum in Theorem \ref{GL3} that agree with the generalized Lindel\"of hypothesis. This restriction forces the length of the sum to be larger. He needs $\Delta\geq T^{1/3}$. Thus the implied subconvexity bound is worse.

\subsection{Bounds for triple products}

Let $\psi$ be a Maass form for $\rm{SL}(2,\Zz)$. By combining Corollary \ref{subconv} for $\phi=\operatorname{sym}^2\psi$ with Ivìc's bound \citep{Ivic2001sums}
$$
L(1/2,f_j)\ll t_j^{1/3+\epsilon},
$$
one gets the bound
$$
L(1/2,\psi\times\psi\times f_j)\ll_{\psi} t_j^{19/12+\epsilon}
$$
We remark that this estimate improves on the bound
$$
L(1/2,\psi_1\times\psi_2\times f_j)\ll_{\Psi_1,\Psi_2} t_j^{5/3+\epsilon},
$$
proved by Bernstein-Reznikov \citep{BR2010subconvexity} in the particular case where $\psi_1=\psi_2$.
However, it seems that an even stronger bound (with the exponent $4/3$) should follow from the methods of Suvitie \citep{suvitie2010short} who proved such a bound in the analog problem for a holomorphic modular forms in the weight aspect.
\subsection{Outline of the proof}

This article belongs to a long line of papers building upon the breakthrough work of Conrey and Iwaniec \citep{CI2000}. An expert in the field will easily be able to recognize the many similarities and the few differences. We also recognize the great deal of influence from \citep{Blomer2012twisted} and \citep{Young2014weyl}.

Our goal is to bound the sum

$$
\sum_{\substack{f_j\in\mathcal{B}\\T-\Delta\leq t_j\leq T+\Delta}}L\left(1/2,\phi\times f_j\right) + (Eis),
$$
%
%
where $(Eis)$ corresponds to the second term in Theorem \ref{GL3}, i.e. the Eisenstein contribution. We use the approximate functional equation for the Rankin-Selberg $L$-function $L\left(1/2,\phi\times f_j\right)$ and we get a sum that looks like
$$
\sum_{\substack{f_j\in\mathcal{B}\\T-\Delta\leq t_j\leq T+\Delta}}\sum_{m^2n\leq T^{3+\epsilon}}\frac{A(n,m)\lambda_j(n)}{(m^2n)^{1/2}}+(Eis),
$$
where $(Eis.)$ stands for a similar term corresponding to the contribution from the Eisenstein series.
By changing the order of summation and using the Kuznetsov formula, one gets to a sum like
$$
\sum_{m^2n\leq T^{3+\epsilon}}\frac{A(n,m)}{(m^2n)^{1/2}}\left(\Delta T\delta_{n,1}+\sum_{\pm}\sum_{c\geq 1}\frac{1}{c}S(n,\pm 1;c)B^{\pm}\left(\frac{4\pi \sqrt{n}}{c}\right)\right),
$$
where $S(m,n;c)$ is a Kloosterman sum, and $B^{\pm}$ is roughly the integral of a Bessel function times some other simple factors along the interval $[T-\Delta,T+\Delta]$. The diagonal terms are easily bounded by $\Delta T^{1+\epsilon}$, which is more than enough.

For the term with the plus sign (the other case is treated similarly), after separating the variables $m$ and $n$ by using the Hecke relations, we are faced with the problem of estimating the sum
$$
N^{-1/2}\sum_{c\geq 1}\frac{1}{c}\sum_{n\sim N}A(n,1)S(n,1;c)B^{+}\left(\frac{4\pi \sqrt{n}}{c}\right),
$$
for $N\leq T^{3+\epsilon}$. Now, an application of the Voronoi summation formula leads roughly to the sum
\begin{equation}\label{sketch1}
\sum_{c\geq 1}c\sum_{n\geq 1}\frac{A(1,\tilde{n})}{\tilde{n}}e\left(\pm \frac{\tilde{n}}{c}\right)\mathcal{W}^{\pm}\left(\frac{N\tilde{n}}{c^3};\frac{\sqrt{N}}{c}\right),
\end{equation}
where
$$
\mathcal{W}^{\pm}(x;D)\approx x^{2/3}\int_{y\asymp 1}B^{+}\left(4\pi D\sqrt{y}\right)e\left(\pm(xy)^{1/3}\right)\frac{dy}{y^{1/3}}.
$$
Actually, once we use the Voronoi summation formula we are faced with a much more complicated exponential sum in place of the simple exponential factor $e\left(\pm\frac{\tilde{n}}{c}\right)$. A delicate study of the intervening sum was done by Blomer \citep{Blomer2012twisted}. Although elementary his argument is rather intricate and we are glad to directly quote his calculations here.

At this point we need a stationary phase analysis of the above integral. Here we use results of \citep{huang2016hybrid}, that were largely based on similar calculations from \citep{Young2014weyl}. The final outcome of this analysis is that $\mathcal{W}^{\pm}(x;D)$ is very small unless $D\gg T$, and in such an event, we have
$$
\mathcal{W}^{\pm}(x;D)\approx e\left(\mp\frac{x}{D^2}\right) \times \Delta x^{1/2}\int_{-T}^{T}\lambda(t)\left(\frac{x}{D^2}\right)^{it}dt,
$$
with $|\lambda(t)|\leq 1$. Once we apply these results to the sum \eqref{sketch1}, we arrive at (notice that the exponential factors cancel out!)
$$
\Delta \int_{-T}^{T}\left|\sum_{c\ll T^{1/2}}\sum_{n\ll T^{3/2}}\frac{A(1,n)}{(cn)^{1/2}}\left(\frac{n}{c}\right)^{it}\right|dt.
$$
The final touch is to use some form of the large sieve combined with the Cauchy-Schwarz inequality. By doing as such, we bound the above sum by
$$
\ll \Delta T^{\epsilon}\left(T^{3/2}+T\right)^{1/2}\left(T^{1/2}+T\right)^{1/2}\ll \Delta T^{5/4 +\epsilon},
$$
as we wanted.

Notice that the fact that the size of the variables $c$ and $n$ are very different is somehow responsible to worsening our estimate. If we could rearrange these two variables into two other variables both with size roughly $T$, we would get the bound $\Delta T^{1+\epsilon}$, which is the limit of the method. In our case it is not clear how to employ such a trick, but if we replaced our cusp form $\phi$ by a (maximal or minimal) parabolic Eisenstein series for $\rm{SL}(3,\Zz)$, this would be possible. In fact, the minimal parabolic case corresponds to a particular case of the result of Young \citep{Young2014weyl}.
\section*{Acknowledgements}

I would like to thank Etienne Fouvry, Philippe Michel and Ian Petrow for stimulating conversations on the topic of this paper.

\section{Preliminaries}

We begin by recalling some classical definitions and properties of Maass forms for $\rm{SL}(2,\Zz)$.

\subsection{Maass forms}

Let $\Hh$ be the classical upper half plane and let $\Delta=-y^2\left(\frac{\partial^2}{\partial x^2}+\frac{\partial^2}{\partial y^2}\right)$. Consider the spectral decomposition
$$
L^2(\rm{SL}(2,\Zz)\backslash \Hh)=\Cc\oplus\mathcal{C}\oplus \mathcal{E}.
$$
Here, $\Cc$ is identified with the space of constant functions on $\Hh$, $\mathcal{C}$ is the space of cusp forms and $\mathcal{E}$ is the space of Eisenstein series.

We recall that $\{f_j\}_{j\geq 1}=\mathcal{B}$ is a basis of Hecke-Maass cusp forms for $\mathcal{C}$. Each $f_j\in \mathcal{B}$ has a Fourier decomposition
$$
f_j(z)=2y^{1/2}\sum_{n\neq 0}\rho_j(n)\sqrt{|n|}K_{it_j}(2\pi|n|y)e(nx),
$$
where $K_s$ is the classical $K$-Bessel function.

We can assume that all the $f_j$ are eigenfunctions of all the Hecke operators with eigenvalues given by $\lambda_j(n)$. So that we have the formula
$$
\rho_j(\pm n)=\rho_j(\pm1)\lambda_j(n)n^{-1/2}.
$$
The functions in $\mathcal{C}$ can be further split in even and odd Maass forms according to whether $f_j(-{\bar z})=f_j(z)$ or $f_j(-{\bar z})=-f_j(z)$. 
Finally, we define the spectral weights
\begin{equation}\label{omegaj}
\omega_j=\frac{4\pi}{\cosh(\pi t_j)}|\rho_j(1)|^2.
\end{equation}
\subsection{Eisenstein series}

The Eisenstein series $E(z,s)$ has a Fourier decomposition as follows:
\begin{equation}
E(z,s)=y^{s}+\eta(s)y^{s}+2y^{1/2}\sum_{n\neq 0}\eta(n,s)\sqrt{|n|}K_{s-1/2}(2\pi|n|y)e(nx).
\end{equation}
The Fourier coefficients can be given explicitly \citep[p. 1187-1188]{CI2000} by
$$
\eta(s)=\pi^{1/2}\frac{\Gamma(s-1/2)}{\Gamma(s)}\frac{\zeta(2s-1)}{\zeta(2s)},
$$
and
$$
\eta(n,s)=\pi^{s}\Gamma(s)^{-1}\zeta(2s)^{-1}|n|^{-1/2}\sigma_{s-1/2}(|n|),
$$
where for $n>0$, $s\in\Cc$, 
$$
\sigma_s(n)=\sum_{ad=n}\left(\frac{a}{d}\right)^s.
$$
We also define the spectral weights
\begin{equation}\label{omegat}
\omega(t)=\frac{4\pi}{\cosh(\pi t)}\left|\eta(1,1/2+it)\right|^2.
\end{equation}
\subsection{Kuznetsov formula}
In what follows we recall the Kuznetsov formula which is a key ingredient of our proof. The next lemma is exactly \citep[Eq. (3.17)]{CI2000}.
\begin{lem}\label{Kuznetsov}
For $m,n\geq 1$, $(mn,q)=1$ and any even test function $h$ satisfying the following conditions
\begin{enumerate}[(i)]
    \item $h$ is holomorphic in $|\operatorname{Im}(t)|\leq \frac{1}{2}+\epsilon$,
    \item $h(t)\ll (1+|t|)^{-2-\epsilon}$ in the above strip,
\end{enumerate}
we have the following identitity: 
\begin{multline*}
\sideset{}{^{\prime}}\sum_{f_j\in\mathcal{B}}h(t_j)\omega_j\lambda_j(m)\lambda_j(n)+\frac{1}{4\pi}\int_{-\infty}^{+\infty}h(t)\omega(t)\sigma_{it}(m)\sigma_{it}(n)dt\\
=\frac{1}{2}\delta_{m,n}D + \frac{1}{2}\sum_{\pm}\sum_{c\equiv0\rmod c}\frac{S(m,\pm n;c)}{c}B^{\pm}\left(\frac{4\pi\sqrt{mn}}{c}\right),
\end{multline*}
where $\mathcal{B}$ is a basis of Hecke-Maass forms, and $\sideset{}{^{\prime}}\sum$ means that we only sum over the even forms, $\delta_{m,n}$ is the Kronecker symbol, and

\begin{equation}\label{DBB}
\begin{cases}
D=\frac{2}{\pi}\displaystyle\int_{0}^{+\infty}h(t)\tanh(\pi t)tdt,\\
B^+(x)=2i\displaystyle\int_{-\infty}^{+\infty}J_{2it}(x)\frac{h(t)t}{\cosh(\pi t)}dt,\\
B^-(x)=\frac{4}{\pi}\displaystyle\int_{-\infty}^{+\infty}K_{2it}(x)\sinh(t)h(t)tdt,
\end{cases}
\end{equation}
where $J_{\nu}$ and $K_{\nu}$ are the standard $J$ and $K$ Bessel functions respectively.

\end{lem}

\subsection{$\rm{GL}(3)$ Maass forms and $L$-functions}\label{lfunctions}

Let $\phi$ be a Hecke-Maass form for $\rm{SL}(3,\Zz)$ of type $(\nu_1,\nu_2)$ whose Hecke eigenvalues are $A(n,m)$. We refer the reader to Goldfeld's book \citep{Goldfeld-GLn} for a precise definition and further information about these forms. In this paper we only use the $L$-functions associated to these forms, so our attention will be focused on them.

We consider the $L$-function associated to $\psi$, given by
$$
L(s,\phi):=\sum_{n\geq 1}\frac{A(n,1)}{n^s},
$$
and for a $\rm{GL}(2)$ Hecke-Maass cusp form $f$ with Hecke eigenvalues $\lambda(n)$, we consider the $L$-function of the Rankin-Selberg convolution of $\phi$ and $f$, \textit{i.e.}
\begin{equation}\label{RSL-function}
L(s,\phi\times f)=\sum_{n\geq 1}\frac{A(n,m)\lambda(n)}{(m^2n)^s}.
\end{equation}

\subsubsection{On the coefficients $A(n,m)$}

At some point in our proof, it will be advantageous to separate the variables $m$ and $n$ in $A(n,m)$ and this will be done by means of the Hecke relations. The following is obtained by applying M\"obius inversion to \citep[Theorem 6.4.11]{Goldfeld-GLn}:

\begin{equation}\label{Hecke}
A(n,m)=\sum_{d\mid (m,n)}\mu(d)A\left(\frac{n}{d},1\right)A\left(1,\frac{m}{d}\right).
\end{equation}

We shall also need estimate for the coefficients $A(n,m)$. For this we restrict ourselves to the case where $\phi$ is a symmetric square lift of a $\rm{GL}(2)$ Hecke Maass form. In this case we have both the pointwise bound
$$
A(n,m)\ll(mn)^{7/32+\epsilon},
$$
and the average result
\begin{equation}\label{RS}
\sum_{m\leq X}A(1,m)^2\ll X.
\end{equation}
The first one is a consequence of the fact that $\phi$ comes from the symmetric square lift of $\rm{GL}(2)$ form and the Kim-Sarnak bound, while the second comes from the Rankin-Selberg theory. Combining these estimates with the multiplicativity property of the $A(n,m)$, i.e.
$$
A(n_1n_2,m_1m_2)=A(n_1,m_1)A(n_2,m_2)\text{,  if }(m_1n_1,m_2n_2)=1,
$$
we obtain the upper bound:
\begin{equation}\label{Rankin-Selberg}
\sum_{m\leq X}|A(a,bm)|^2\ll X(ab)^{\frac{7}{16}+\epsilon}.
\end{equation}

\subsubsection{Approximate functional equation}

It is common knowledge among specialists that the value of an $L$- function on the critical line can be expressed by an essentially finite sum. The next two lemmas are a consequence of \citep[Theorem 5.3]{IK-analytic}.
In the following we consider the Langlands parameters
$$
\alpha_1=-\nu_1-2\nu_2+1,\;\;\alpha_2=-\nu_1+\nu_2,\;\;\alpha_3=-2\nu_1+\nu_2-1.
$$
\begin{lem}\label{AFE}
We have the following identities:

$$
L(1/2,\phi\times f_j\times\chi)=2\displaystyle\sum_{m=1}^{+\infty}\displaystyle\sum_{n=1}^{+\infty}\frac{A(n,m)\lambda_j(n)\chi(n)}{(m^2n)^{1/2}}V_{t_j}\left(\frac{m^2n}{q^2}\right),
$$
and
$$
\left|L(1/2+it,\phi\times\chi)\right|^2=2\displaystyle\sum_{m=1}^{+\infty}\displaystyle\sum_{n=1}^{+\infty}\frac{A(n,m)\sigma_{it}(n)\chi(n)}{(m^2n)^{1/2}}V_{t}\left(\frac{m^2n}{q^2}\right),
$$
where
$$
V_t(y)=\frac{1}{2\pi i}\int_{(3)}(\pi^2 y)^{-u}\prod_{\pm}\prod_{i=1}^{3}\frac{\Gamma\left(\frac{1/2\pm it+u-\alpha_i}{2}\right)}{\Gamma\left(\frac{1/2\pm it-\alpha_i}{2}\right)}e^{u^2}\frac{du}{u}.
$$
\end{lem}
The next lemma shows that all the sums in Lemma \ref{AFE} are essentially bounded. It also describes explicitly the dependency of $V_t$ on the variable $t$.

\begin{lem}\label{truncation}
\begin{enumerate}[(i)]
\item For $k\geq 0$
$$
y^kV_t^{(k)}(y)\ll \left(1+\frac{y}{(1+|t|)^3}\right)^{-A},
$$
and
$$
y^kV_t^{(k)}(y)=\delta_k+O\left(\left(\frac{y}{(1+|t|)^3}\right)^{\alpha}\right),
$$
for any $0<\alpha\leq \frac{1}{3}\displaystyle\min_{1\leq i\leq 3}(1/2-|\Re(\alpha_i)|)$, where $\delta_0=1$ and $\delta_k=0$ otherwise.

\item
For $1<L\ll T^{\epsilon}$, $\epsilon>0$, and $|t-T|\ll T^{1-2\epsilon}$, we have the following approximation
\begin{multline*}
V_t(y)=\sum_{k=0}^{K/2}\sum_{\ell=0}^{K/2}t^{-2k}\left(\frac{t^2-T^2}{T^2}\right)^{\ell}V_{k,\ell}\left(\frac{y}{T^3}\right)+ O(y^{-\epsilon}(1+|T|)^{\epsilon}e^{-L})\\
+O\left(\left(\frac{1+|t-T|}{T}\right)^{K+1}\left(1+\frac{y}{T^3}\right)^{-A}\right),
\end{multline*}
where
\begin{equation}\label{Vkl}
V_{k,\ell}(y)=\frac{1}{2\pi i}\int_{\epsilon-iL}^{\epsilon+iL}P_{k.\ell}(u)(2\pi)^{-3u}y^{-u}e^{u^2}\frac{du}{u}.
\end{equation}
\end{enumerate}
\end{lem}

\begin{proof}
For a proof, one can check \citep[Lemma 2.2]{huang2016hybrid}.
\end{proof}

We shall use this result with $U=(\log T)^2$, so that the first error term is $O_{A}(T^{-A})$ for every $A>0$.

\begin{rmk}
Thanks to the works of Luo-Rudnick-Sarnak \citep{LRS99} on the Ramanujan conjecture, one can take any $\alpha\leq \frac{1}{10}$ in the lemma above.
\end{rmk}

\subsection{Voronoi summation formula}

In the next lemma we recall the Voronoi summation formula for $\rm{SL}(3,\Zz)$-Maass forms. The Voronoi formula is a generalization of the classical Poisson summation formula and is in a certain sense equivalent to the functional equation for $L(s,\phi\times\chi)$ where $\chi$ is a multiplicative character. The first version of this formula was obtained by Miller and Schmid \citep{MS2006automorphic}. The version we give here is \citep[Lemma 3]{Blomer2012twisted}

\begin{lem}\label{voronoi}

Let $w:(0,+\infty)\rightarrow\Cc$ be a smooth function with compact support. Let $\widehat{w}(s)$ denote its Melin transform and let
$$
G^{\pm}(x):=\prod_{i=1}^{3}\frac{\Gamma\left(\frac{s+\alpha_i}{2}\right)}{\Gamma\left(\frac{1-s-\alpha_i}{2}\right)}\pm \frac{1}{i}\prod_{i=1}^{3}\frac{\Gamma\left(\frac{1+s+\alpha_i}{2}\right)}{\Gamma\left(\frac{2-s-\alpha_i}{2}\right)}.
$$
Then we have
\begin{multline}
\sum_{n=1}^{+\infty}A(n,m)e\left(\frac{{\bar d}n}{c}\right)w(n)\\
=\frac{\pi^{3/2}}{2}c\sum_{\pm}\sum_{n_1\mid cm}\sum_{n_2=1}^{+\infty}\frac{A(n_1,n_2)}{n_1n_2}S\left(dm,\pm n_2;\frac{mc}{n_1}\right)\mathcal{W}\left(\frac{n_1^2n_2}{c^3m}\right),
\end{multline}
where
\begin{equation}\label{Wtransform}
\mathcal{W}^{\pm}(x)=\frac{x}{2\pi i}\int_{(3)}G^{\pm}(s)\widehat{w}(y)dy.
\end{equation}
\end{lem}

\section{Initial steps for Theorem \ref{GL3}}

We want to prove

$$
\sum_{\substack{f_j\in\mathcal{B}\\T-\Delta\leq t_j\leq T+\Delta}}L\left(1/2,\phi\times f_j\right)+\frac{1}{4\pi}\int_{T-\Delta}^{T+\Delta}\left|L\left(1/2+it,\phi\right)\right|^2dt\ll \Delta T^{5/4+\epsilon}.
$$

We shall consider the spectrally normalized forms

$$
\mathcal{M}:=\sideset{}{{}^{\prime}}\sum_{\substack{f_j\in\mathcal{B}\\T-\Delta\leq t_j\leq T+\Delta}}\omega_jL\left(\frac{1}{2},\phi\times f_j\right)+\frac{1}{4\pi}\int_{T-\Delta}^{T+\Delta}\omega(t)\left|L\left(\frac{1}{2}+it,\phi\right)\right|^2dt,
$$
where $\omega_j$ and $\omega(t)$ are given by \eqref{omegaj} and \eqref{omegat} respectively. There is not such a big loss because we have the inequalities
$$
\omega_j\gg t_j^{-\epsilon},\,\,\omega(t)\gg t^{-\epsilon}.
$$
The first of these upper bounds is conteined in \citep[Theorem 8.3]{iwaniec-spectral} and the second one is a classical result.

Hence we need to prove
\begin{equation}\label{whatwillprove}
\mathcal{M}\ll \Delta T^{5/4+\epsilon}.
\end{equation}

We want to estimate $\mathcal{M}$ by means of the Kuznetsov formula. To do so, we need to consider a smooth variant of $\mathcal{M}$. Precisely, we let
\begin{equation}\label{h=}
h(t)=\frac{1}{\cosh\left(\frac{t-T}{\Delta}\right)}+\frac{1}{\cosh\left(\frac{t+T}{\Delta}\right)}.
\end{equation}
Here one could take other nice holomorphic even functions such that $h(t)\gg 1$ in the region $|t-T|\leq \Delta$ but we take this particular one so that we can directly quote results from \citep{Young2014weyl} and \citep{huang2016hybrid}. 
Since $h(t)\gg 1$ in $[T-\Delta,T+\Delta]$, it follows that
\begin{multline*}
\mathcal{M}\ll\sideset{}{{}^{\prime}}\sum_{f_j\in\mathcal{B}}h(t_j)\omega_jL\left(\frac{1}{2},\phi\times f_j\right)+\\
\frac{1}{4\pi}\int_{-\infty}^{+\infty}h(t)\omega(t)\left|L\left(\frac{1}{2}+it,\phi\right)\right|^2dt.
\end{multline*}
Applying the approximate functional equation (see Lemma \ref{AFE}), the right-hand side above equals

\begin{multline*}
\sideset{}{{}^{\prime}}\sum_{f_j\in\mathcal{B}}h(t_j)\omega_j\sum_{m=1}^{+\infty}\sum_{n=1}^{+\infty}\frac{A(n,m)\lambda_j(n)}{(m^2n)^{1/2}}V_{t_j}\left(m^2n\right)\\
+\frac{1}{4\pi}\int_{-\infty}^{+\infty}h(t)\omega(t)\sum_{m=1}^{+\infty}\sum_{n=1}^{+\infty}\frac{A(n,m)\sigma_{it}(n)}{(m^2n)^{1/2}}V_{t}\left(m^2n\right)dt,
\end{multline*}
By Lemma \ref{truncation}, there exists $K>0$ for which is enough to bound

\begin{multline}\label{Mdag}
\mathcal{M}^{\dag}:=\sum_{m=1}^{+\infty}\sum_{n=1}^{+\infty}\frac{A(n,m)}{(m^2n)^{1/2+u}}V\left(\frac{m^2n}{T^3}\right)\\
\times\left(\sideset{}{{}^{\prime}}\sum_{f_j\in\mathcal{B}}h_{k,\ell}(t_j)\omega_j\lambda_j(n)+\frac{1}{4\pi}\int_{-\infty}^{+\infty}h_{k,\ell}(t)\omega(t)\sigma_{it}(n)dt\right),
\end{multline}
uniformly in $u\in [\epsilon-i(\log T)^2,\epsilon+i(\log T)^2]$ and $0\leq k,\ell\leq K$, where $V=V_{k,\ell}$ is of the form \eqref{Vkl}, and
\begin{equation}\label{hkl}
h_{k,\ell}=t^{-2k}T^{-2\ell}(t^2-T^2)^{\ell}h(t).
\end{equation}
In the following we will give all the details only in the case where $k=\ell=0$. The other cases can be handled similarly and amount to smaller order terms.

We are now in a perfect position to use the Kuznetsov formula to the terms between parenthesis in \eqref{Mdag}. We thus obtain

\begin{multline*}
\mathcal{M}^{\dag}=\sum_{m=1}^{\infty}\sum_{n=1}^{\infty}\frac{A(n,m)\chi(n)}{(m^2n)^{1/2+u}}V\left(\frac{m^2n}{T^3}\right)\Bigg(\frac{1}{2}\delta_{n,1}D+\\
\frac{1}{2}\sum_{c\geq 1}\frac{1}{c}\sum_{\pm}S(n,\pm 1;c)B^{\pm}\left(\frac{4\pi\sqrt{n}}{c}\right)\Bigg),
\end{multline*}
where $D$ and $B^{\pm}(x)$ are as in \eqref{DBB}.

It follows from \eqref{h=} that $D\ll \Delta T^{1+\epsilon}$ and hence, by Cauchy-Schwarz and \eqref{RS}, we have
\begin{multline*}
\sum_{m=1}^{\infty}\sum_{n=1}^{\infty}\frac{A(n,m)}{(m^2n)^{1/2+u}}V\left(\frac{m^2n}{T^3}\right)\delta_{n,1}D\\
\ll \Delta T^{1+\epsilon}\sum_{m\ll T^{3/2+\epsilon}}\frac{1}{m}\ll \Delta T^{1+\epsilon},
\end{multline*}
which suffices for our purposes. Now we need to estimate the off-diagonal terms. Let
$$
S_{\sigma}:=\sum_{m=1}^{\infty}\sum_{n=1}^{\infty}\frac{A(n,m)}{(m^2n)^{1/2+u}}V\left(\frac{m^2n}{T^3}\right)\sum_{c\equiv 0\rmod q}\frac{1}{c}S(n,\sigma;c)B^{\sigma}\left(\frac{4\pi\sqrt{n}}{c}\right),
$$
for $\sigma=\pm$. In the following, we must separate the variables $m$ and $n$ in $A(n,m)$ by means of the Hecke relations (see \eqref{Hecke}). This together with a change of variables and a dyadic decomposition on the $n$ variable shows that
$$
S_{\sigma}\ll \underset{\delta^3m^2\leq T^{3+\epsilon}}{\sum\sum}\frac{|A(1,m)|}{\delta^{3/2}m}\displaystyle\sup_{N\ll \frac{T^{3+\epsilon}}{\delta^3m^2}}\frac{\left|S_{\sigma}(N;\delta)\right|}{N^{1/2}} + O(T^{-A}),
$$
where
\begin{equation}
S_{\sigma}(N;\delta)=\sum_{c=1}^{+\infty}\frac{1}{c}\sum_{n}A(n,1)S(\delta n,\sigma;c)w_{\sigma}\left(\frac{n}{N};\frac{\sqrt{\delta N}}{c}\right),
\end{equation}
where
$$
w_{\sigma}(y;D):=w(y)y^{-u}B^{\sigma}(4\pi D\sqrt{y}),
$$
for some smooth funcion $w$ with compact support.

We must prove the following

\begin{prop}\label{mustprove}
Let $S_{\sigma}(N;\delta)$ be as above. Then we have the inequality
\begin{equation*}
S_{\sigma}(N;\delta)\ll \Delta \delta^{1/2}N^{1/2}T^{5/4+\epsilon},
\end{equation*}
uniformly for $\delta^3N\ll T^{3+\epsilon}$
\end{prop}

\subsection{Applying the Voronoi formula}

The next step is to apply the Voronoi summation formula on the $n$ variable. Opening the Kloosterman sum and using Lemma \ref{voronoi}, we obtain

\begin{multline}
S_{\sigma}(N;\delta)=\frac{\pi^{3/2}}{2}\sum_{\pm}\sum_{c=1}^{+\infty}\frac{1}{\delta_0}\sum_{n_1 \mid c_1}\sum_{n_2=1}^{\infty}\frac{A(n_1,n_2)}{n_1n_2}\\
\times \mathcal{W}^{\pm}_{\sigma}\left(\frac{Nn_1^2n_2}{c_1^3};\frac{\sqrt{\delta N}}{c}\right)\mathcal{T}^{\pm,\sigma}_{\delta,n_1,n_2}(c),
\end{multline}
where $\delta_0=(\delta,c)$, $\delta'=\delta/\delta_0$, $c'=c/\delta_0$,
$$
\mathcal{T}^{\pm,\sigma}_{\delta,n_1,n_2}(c):=\sideset{}{^{\ast}}\sum_{d\rmod{c}}e\left(\frac{\sigma d}{c}\right)S\left(\overline{\delta'}d,\pm n_2;c'/n_1\right),
$$
and $\mathcal{W}^{\pm}_{\sigma}(x;D)$ is defined as in \eqref{Wtransform} with $w(.)$ replaced by $w_{\sigma}(.,D)$.\\

In the next sections we shall study the integral transform $\mathcal{W}_{\sigma}^{\pm}(.,.)$ and the exponential sum $\mathcal{T}^{\pm,\sigma}_{\delta,n_1,n_2}(c)$.

\section{Stationary phase method}

In this section we deduce a particularly nice formula for $\mathcal{W}^{\pm}$. Precisely, we show that after multiplying it by a suitable oscillating factor, we can write it roughly as an integral of length $T$ of an archimedean character, which will be perfect for the application of the large sieve later on.

The next lemma gives the asymptotic behavior of the integral transform $\mathcal{W}^{\pm}$. This can be found in some form in \citep{li2009central}, but here we use a slightly more precise form given by Blomer (see \citep[Lemma 6]{Blomer2012twisted}).

\begin{lem}\label{Besseldecomposition}
Let $K>0$. There exist constants $\gamma_{\ell}$ depending only on the Langlands parameters $\alpha_i$ of $\phi$ such that for any compactly supported function $w$ and $x\geq 1$, we have
$$
\mathcal{W}^{\pm}(x)=x\int_{0}^{+\infty}w(y)\sum_{j=1}^{K}\frac{\gamma_{\ell}}{(xy)^{\ell/3}}e\left(\pm (xy)^{1/3}\right)dy +O(x^{1-K/3}),
$$
where the implied constant depend only on the $\alpha_i$. $\|w\|_{\infty}$ and $K$.
\end{lem}

The next lemma combines Lemma \ref{Besseldecomposition} with stationary phase arguments due to Young (see \citep[Lemma 8.1]{Young2014weyl}). The details were carried out by Huang (see \citep[Lemma 4.4]{huang2016hybrid}).

\begin{lem}\label{integralrep}
Let $x\gg T^{-B}$ for some large but fixed $B$. Let
\begin{equation}\label{W-tilde}
\widetilde{\mathcal{W}}_{\sigma}^{\pm}(x;D):=e\left(\mp\sigma\frac{x}{D^2}\right)\mathcal{W}(x;D).
\end{equation}
We have $\widetilde{\mathcal{W}^{\pm}_{\sigma}}(x;D)\ll T^{-A}$, unless
\begin{align}
\begin{cases}
D\gg T\Delta^{1-\epsilon}, &\text{if }\sigma=1,\\
D\asymp T, &\text{if }\sigma=-1.\\
\end{cases}
\end{align}
If $x\gg T^{\epsilon}$, we have
\begin{equation}\label{Wtilde}
\widetilde{\mathcal{W}^{\pm}_{\sigma}}(x;D)=\Delta x^{5/6}\sum_{j=1}^{K}\frac{\gamma_{\ell}}{x^{\ell/3}}L_j(x;D) +O_A(T^{-A}),
\end{equation}
where $L_j$ is a function that takes the form
$$
L_j(x;D)=\int_{|t|\ll U}\lambda_{X,D}(t)\left(\frac{x}{D^2}\right)^{it}dt,
$$
with the following parameters. Here $\lambda_{X,T}(t)\ll 1$ does not depend on $x$ or $D$. If $\sigma=1$, then $U=T^2/D$ and $L_j$ vanishes unless
$$
X\asymp D^3,\text{   and   }D\gg \Delta T^{1-\epsilon}.
$$
If $\sigma=-1$, then $U=T^{2/3}X^{1/3}D^{-2/3}$ and $L_j$ vanishes unless
$$
X\ll D^3\Delta^{-3+\epsilon},\text{   and   }D\asymp T.
$$
\end{lem}
\section{Treatment of the exponential sum $\mathcal{T}^{\pm,\sigma}_{\delta,n_1,n_2}(c)$}

Before we use the results of the last section, we must first write a formula for $\mathcal{T}^{\pm,\sigma}_{\delta,n_1,n_2}(c)$ that has the dependence in $n_2$ in a rather nice explicit way. In particular, we will find the term $e\left(\mp\sigma\frac{\delta_0n_1^2n_2}{\delta'c'}\right)$ necessary to form $\widetilde{\mathcal{W}^{\pm}}(\ldots)$ (see \eqref{Wtilde}).

This is a very delicate calculation and while doing it many new variables will arise. But almost all of them play no important role in the argument.

We remark that our definition of $\mathcal{T}^{\pm,\sigma}_{\delta,n_1,n_2}(c)$ can be seen as a special case of $\mathcal{T}^{\pm,\sigma}_{\delta,c_1,n_1,n_2}(c,q)$ in \citep[p. 1407]{Blomer2012twisted}. Indeed we have
$$
\mathcal{T}^{\pm,\sigma}_{\delta,n_1,n_2}(c)=\frac{1}{c}\mathcal{T}^{\pm,\sigma}_{\delta,\frac{c}{(\delta.c)},n_1,n_2}(c,1).
$$
Thus, \citep[Lemma 12]{Blomer2012twisted} becomes

\begin{lem}
Let $\delta_0=(\delta,c)$, $\delta'=\delta/\delta_0$ and $c'=c/\delta_0$.

If $(\delta_0,c')=1$, then

\begin{multline}\label{exp-lem-blomer}
\mathcal{T}^{\pm,\sigma}_{\delta,n_1,n_2}(c)=e\left(\mp\sigma\frac{\overline{\delta'}\delta_0n_1^2n_2}{c'}\right)\frac{\varphi(c'/n_1)}{\varphi(c')}\frac{\mu(\delta_0)}{\delta_0}c\\
\times\sum_{\substack{d_2'f_1f_2=c'\\(d_2',f_1n_1n_2)=1\\(f_1,f_2)=1,\,f_2\mid n_1}}\frac{\mu(f_1)^2\mu(f_2)}{f_1}e\left(\pm\sigma\frac{\overline{\delta'd_2'}\delta_0f_2(n_1')^2n_2}{f_1}\right),
\end{multline}
where $n_1'=n_1/f_2$ and $\mathcal{T}^{\pm,\sigma}_{\delta,n_1,n_2}(c)=0$ otherwise.
\end{lem}

By a double application of the classical formula $e\left({\bar a}/{b}+{\bar b}/{a}\right)=e\left({1}/{ab}\right)$, we see that the product of of the two exponentials on the right-hand side of \eqref{exp-lem-blomer} equals
$$
e\left(\mp\sigma\frac{\delta_0n_1^2n_2}{\delta'c'}\right)e\left(\pm\sigma\frac{\overline{d_2'}\delta_0f_2(n_1')^2n_2}{\delta'f_1}\right).
$$
The first factor combines perfectly with $\mathcal{W}^{\pm}_{\sigma}\left(\frac{Nn_1^2n_2}{c_1^3};\frac{\sqrt{\delta N}}{c}\right)$ to form $\widetilde{\mathcal{W}^{\pm}_{\sigma}}\left(\frac{Nn_1^2n_2}{c_1^3};\frac{\sqrt{\delta N}}{c}\right)$.

We conclude that
\begin{multline*}
S_{\sigma}(N;\delta)=\frac{\pi^{3/2}}{2}\sum_{\pm}\sum_{\delta_0\delta'=\delta}\frac{\mu(\delta_0)}{\delta_0}\sum_{\substack{d_2',f_1,f_2\\(\delta,d_2'f_1f_2)=1\\(f_1,f_2)=1,\,(d_2',f_1f_2)=1}}\sum_{\substack{n_1\mid f_1\\(d_2',n_1')=1}}\\
\times\sum_{\substack{n_2\\(d_2',n_2)=1}}\mu(f_1)^2\mu(f_2)d_2'\frac{\varphi(d_2'f_1/n_1)}{\varphi(d_2'f_1f_2)}\frac{A(f_2n_1,n_2)}{n_1'n_2}\\
\times e\left(\pm\sigma\frac{\overline{d_2'}\delta_0f_2(n_1')^2n_2}{\delta'f_1}\right)\widetilde{\mathcal{W}^{\pm}_{\sigma}}\left(\frac{N(n_1')^2n_2}{(d_2'f_1)^3f_2};\frac{\sqrt{\delta N}}{\delta_0d_2'f_1f_2}\right).
\end{multline*}
We write $f_1=gn_1'$. Then we have
\begin{multline*}
S_{\sigma}(N;\delta)=\frac{\pi^{3/2}}{2}\sum_{\pm}\sum_{\delta_0\delta'=\delta}\frac{\mu(\delta_0)}{\delta_0}\sum_{\substack{d_2',f_2,g,n_1'\\(\delta,d_2'f_2gn_1')=1\\(f_2,gn_1')=1\\(d_2',f_2gn_1')=1}}\sum_{\substack{n_2\\(d_2',n_2)=1}}\frac{\mu(gn_1')^2\mu(f_2)d_2'}{\varphi(f_2n_1')n_1'}\\
\times\frac{A(f_2n_1',n_2)}{n_2}e\left(\pm\sigma\frac{\overline{d_2'}\delta_0f_2n_1'n_2}{\delta'g}\right)\widetilde{\mathcal{W}^{\pm}_{\sigma}}\left(\frac{Nn_2}{(d_2'g)^3f_2n_1'};\frac{\sqrt{\delta N}}{\delta_0d_2'f_2gn_1'}\right).
\end{multline*}
Since $(\delta_0d_2'f_2n_1',\delta'g)=1$. Let
$$
s=(\delta'g,n_2),\,\,n_2=n_2's,\,\,(n_2',\delta'g/s).
$$
We deduce that
\begin{multline}\label{lastbeforeStPh}
S_{\sigma}(N;\delta)=\frac{\pi^{3/2}}{2}\sum_{\pm}\sum_{\delta_0\delta'=\delta}\frac{\mu(\delta_0)}{\delta_0}\sum_{\substack{f_2,g,n_1'\\(\delta,f_2gn_1')=1\\(f_2,gn_1')=1}}\frac{\mu(gn_1')^2\mu(f_2)}{\varphi(f_2n_1')n_1'}\sum_{s\mid \delta'g}\frac{1}{s}\\
\sum_{\substack{d_2'\\(d_2',\delta f_2gn_1')=1}}\sum_{\substack{n_2'\\(n_2',d_2'\delta'g/s)=1}}\frac{A(f_2n_1',n_2's)d_2'}{n_2'}e\left(\pm\sigma\frac{\overline{d_2'}\delta_0f_2n_1'n_2'}{\delta'g/s}\right)\\
\times\widetilde{\mathcal{W}^{\pm}_{\sigma}}\left(\frac{Nn_2'}{(d_2'g)^3f_2n_1'};\frac{\sqrt{\delta N}}{\delta_0d_2'f_2gn_1'}\right).
\end{multline}
Let 
$$
x:=\frac{Nn_2}{(d_2'g)^3f_2n_1'},\text{ and }D:=\frac{\sqrt{\delta N}}{\delta_0d_2'f_2gn_1'}.
$$
We note that by Lemma \ref{integralrep}, the contribution to the right-hand side of \eqref{lastbeforeStPh} of the terms where 
$$
D\ll T^{1-\epsilon}.
$$
is negligible. This implies that we can impose $d_2'\ll T^{\epsilon}\frac{\sqrt{\delta N}}{T\delta_0f_2gn_1'}$. And hence

$$
x\gg T^{3-\epsilon}\frac{\delta_0^3(f_2n_1')^2n_2'}{(\delta^3N)^{1/2}}\gg T^{\frac{3}{2}-\epsilon},
$$
since we are assuming $\delta^3N\leq T^{3+\epsilon}$. That means that we can apply the second part of Lemma \ref{integralrep} to the remaining terms in the right-hand side of \eqref{lastbeforeStPh}. uppon making a dyadic decomposition of the variables $d_2'$ and $n_2'$, we deduce that except for a negligible term ($O(T^{-A})$), $S_{\sigma}(N;\delta)$ is

\begin{multline}\label{almostthere}
\ll \Delta T^{\epsilon}\sum_{\pm}\sum_{\delta_0\delta'=\delta}\frac{1}{\delta_0}\sum_{\substack{f_2,g,n_1',r\\(\delta,f_2gn_1')=1\\(f_2,gn_1')=1\\(r,\delta f_2gn_1')}}\sum_{s\mid \delta'g}\frac{1}{(f_2g)^{3/2}(n_1')^{5/2}rs^{1/2}}\\
\times\sup_{D_2,N_2}\frac{N^{1/2}}{D_2^{1/2}N_2^{1/2}}\int_{|t|\ll U}\Bigg|\sum_{\substack{D_2<d_2'\leq D_2\\(d_2',\delta f_2gn_1')=1}}\sum_{\substack{N_2<n_2'\leq 2N_2\\(n_2'',d_2''\delta'g/s)=1}}\alpha(d_2'')\beta(n_2'')\\
\times A(f_2n_1',n_2''rs)e\left(\pm\sigma\frac{\overline{d_2''}\delta_0f_2n_1'n_2''}{\delta'g/s}\right)\left(\frac{n_2''}{d_2''}\right)^{it}\Bigg|dt,
\end{multline}
where 

\begin{equation}\label{D2N2}
\begin{cases}
1\ll D_2 \ll \frac{(\delta N)^{1/2}}{T\delta_0f_2gn_1'r},\\
1\ll N_2 \ll \frac{(\delta^3 N)^{1/2}}{\delta_0^3(f_2n_1')^2rs},\\
U\ll T.
\end{cases}
\end{equation}
Note that in particular, we have
\begin{equation}\label{r<}
r\ll \frac{(\delta N)^{1/2}}{T}\ll T^{1/2+\epsilon}.
\end{equation}

\section{The large sieve and end of the proof}

The final step consists in applying some version of the large sieve to the integral in the right-hand side of \eqref{almostthere}, but before we do so we need to decompose the exponential factor     as a combination of characters weighted by Gauss sums. All this is accomplished in the following lemma. We have

\begin{lem}\label{largesieve}
Let $a,b,c\in\Zz$ such that $(ab,c)=1$. Let $D,N\geq 1$. Suppose $\alpha_d, \beta_n\in \Cc$ and $U\geq 1$. Then for any $\epsilon>0$, we have
\begin{multline*}
\int_{-U}^{U}\Big|\underset{(dn,c)=1}{\sum_{d\sim D}\sum_{n\sim N}}\alpha_d\beta_n e\left(\frac{\overline{ad}bn}{c}\right)\left(\frac{n}{d}\right)^{it}\Big|dt\\
\ll c^{1/2}(D+U)^{1/2}(N+U)^{1/2}\|\alpha\|_2\|\beta\|_2,
\end{multline*}
where
$$
\|\alpha\|_2:=\left(\sum_{d}|\alpha_d|^2\right)^{1/2},\;\;\|\beta\|_2:=\left(\sum_{n}|\beta_n|^2\right)^{1/2}.
$$
\end{lem}
\begin{proof}
For every multiplicative character modulo $c$, we let
$$
\tau(\chi):=\sum_{x\rmod{c}}\chi(x)e\left(\frac{x}{c}\right)
$$
be the Gauss sum associated to $\chi$. Then by orthogonality of characters, the integral on the left-hand side is

$$
\ll\frac{1}{\varphi(c)}\sum_{\chi\rmod{c}}\left|\tau( \chi)\right|\int_{-U}^{U}\Big|\underset{(dn,c)=1}{\sum_{d\sim D}\sum_{n\sim N}}\alpha_d\beta_n \chi(d/n)\left(\frac{n}{d}\right)^{it}\Big|dt.
$$
By using the Cauchy-Schwarz inequality we see that the inner integral is bounded by
$$
\left(\int_{-U}^{U}\Big|\sum_{\substack{d\sim D\\(d,c)=1}}\alpha_d\chi(d)d^{it}\Big|^2dt\right)^{1/2}\left(\int_{-U}^{U}\Big|\sum_{\substack{n\sim N\\(n,c)=1}}\beta_n{\bar \chi}(n)n^{it}\Big|^2dt\right)^{1/2}.
$$
Now the large sieve (see \citep[Theorem 2]{gallagher1970large}) implies that the above expression is
$$
\ll (D+U)^{1/2}(N+U)^{1/2}\|\alpha\|_2\|\beta\|_2.
$$
The Lemma now follows from the classical bound for the Gauss sums.
\end{proof}

Using Lemma \ref{largesieve} for for the integral in \eqref{almostthere}, we may bound it by

$$
\left(\frac{\delta'g}{s}\right)^{1/2}D_2^{1/2}(D_2+U)^{1/2}(N+U)^{1/2}\left(\sum_{n_2''\sim N_2}A(f_2n_1',n_2''rs)^2\right)^{1/2}.
$$
We estimate the sum over $n_2''$ by means of the Rankin-Selberg bound (see \eqref{Rankin-Selberg}) and we recall that because of \eqref{D2N2}, and the inequality $\delta^3N\leq T^{3+\epsilon}$, we have
$$
D_2\ll T^{1/2+\epsilon}/r,\,\,\,N_2\ll T^{3/2+\epsilon}/r,\,\,\,U\ll T.
$$
Putting everything together and applying it to \eqref{almostthere}, we get (recall \eqref{r<})
$$
\frac{S_{\sigma}(N;\delta)}{N^{1/2}}\ll T^{\epsilon}\Delta\delta^{1/2}\sum_{r\ll T^{1/2+\epsilon}}r^{-25/32}\left(\frac{T^{1/2}}{r}+T\right)^{1/2}\left(\frac{T^{3/2}}{r}+T\right)^{1/2},
$$
since all the other sums are convergent. From this we deduce

\begin{align*}
\frac{S_{\sigma}(N;\delta)}{N^{1/2}}&\ll \Delta T^{\epsilon}\sum_{r\ll T^{1/2}}\delta^{1/2}r^{-25/32}\left(\frac{T^{5/4}}{r^{1/2}}+T\right)\\
&\ll \Delta\delta^{1/2}\left(T^{5/4+\epsilon}+T^{39/32+\epsilon}\right).
\end{align*}
We have thus proved Proposition \ref{mustprove} and hence Theorem \ref{GL3}.

\bibliographystyle{plain}
\bibliography{references}

\end{document}